\documentclass[12pt]{amsart}

\usepackage{amssymb,amsfonts}
\usepackage{amsthm}
\usepackage{graphicx}
\usepackage[all,arc]{xy}
\usepackage{hyperref}
\hypersetup{colorlinks,allcolors=blue}
\usepackage{enumerate}
\usepackage{bbm}
\usepackage{dsfont}
\usepackage{mathrsfs}
\usepackage{tikz-cd}
\usetikzlibrary{shapes}
\usepackage{import}
\usepackage{float}
\restylefloat{table}

\usepackage[left=1in,top=1in,right=1in,bottom=1in]{geometry}

\usepackage{mathtools}
\newtheorem{thm}{Theorem}[section]
\newtheorem*{thm*}{Theorem}

\newtheorem{innercustomthm}{Theorem}
\newenvironment{customthm}[1]
  {\renewcommand\theinnercustomthm{#1}\innercustomthm}
  {\endinnercustomthm}

\newtheorem{prop}[thm]{Proposition}
\newtheorem{lem}[thm]{Lemma}
\newtheorem{conj}[thm]{Conjecture}

\theoremstyle{definition}
\newtheorem{defn}[thm]{Definition}

\newtheorem*{thm1.2}{\textrm{Theorem 1.2}}

\theoremstyle{remark}

\newcommand{\Z}{\mathbb{Z}}
\newcommand{\Q}{\mathbb{Q}}
\renewcommand{\P}{\mathbb{P}}

\newcommand{\Aut}{\operatorname{Aut}}

\newcommand{\val}{\operatorname{val}}

\newcommand{\ch}{\operatorname{ch}}

\newcommand{\overbar}[1]{\mkern 1.5mu\overline{\mkern-1.5mu#1\mkern-1.5mu}\mkern 1.5mu}

\def\C{\mathbb{C}}
\def\M{\mathcal{M}}
\def\P{\mathbf{P}}
\def\bbP{\mathbb{P}}
\def\Q{\mathbb{Q}}
\def\R{\mathbb{R}}
\def\T{\mathbf{T}}

\def\calC{\mathcal{C}}
\def\calM{\mathcal{M}}

\def\calP{\mathcal{P}}

\def\bfx{\boldsymbol{x}}
\def\bfy{\boldsymbol{y}}
\def\scrP{\mathscr{P}}

\newcommand{\Mbar}{\overbar{\mathcal{M}}}

\newcommand{\RRT}{\mathrm{RRT}}

\newcommand{\defeq}{\vcentcolon=}

\makeatletter
\let\c@equation\c@thm
\makeatother
\numberwithin{equation}{section}

\graphicspath{ {images/} }
\title[Cut-and-paste invariants of moduli of relative stable maps to $\bbP^1$]{Moduli of relative stable maps to $\mathbb{P}^1$: cut-and-paste invariants}

\author{Siddarth Kannan}\address{Department of Mathematics, Brown University, Providence, RI 02906}
\email{\url{siddarth_kannan@brown.edu}}

\begin{document}

\begin{abstract}
We study constructible invariants of the moduli space $\Mbar(\bfx)$ of stable maps from genus zero curves to $\bbP^1$, relative to $0$ and $\infty$, with ramification profiles specified by ${\bfx \in \Z^n}$. These spaces are central to the enumerative geometry of $\bbP^1$, and provide a large family of birational models of the Deligne--Mumford--Knudsen moduli space $\Mbar_{0,n}$. For the sequence of vectors $\bfx$ corresponding to maps which are maximally ramified over $0$ and unramified over $\infty$, we prove that a generating function for the topological Euler characteristics of these spaces satisfies a differential equation which allows for its recursive calculation. We also show that the class ${[\Mbar(\bfx)] \in K_0(\mathsf{Var}/\C)}$ of the moduli space in the Grothendieck ring of varieties is constant as $\bfx$ varies within a fixed chamber in the resonance decomposition of $\Z^n$.  We conclude by suggesting several further directions in the study of these spaces, giving conjectures on (1) the asymptotic behavior of the Euler characteristic and (2) a potential chamber structure for the Chern numbers.
\end{abstract}
	
\maketitle\thispagestyle{empty}
\section{Introduction}
Given a vector of nonzero integers $\bfx = (x_1, \ldots, x_n) \in \Z^n$ with $\sum x_i = 0$, one may consider the moduli space of maps
\[f : (\bbP^1, p_1, \ldots, p_n) \to (\bbP^1, 0, \infty) \]
where the $p_i$ are distinct, $f^{-1}(0) = \{p_i \mid x_i > 0\}$, $f^{-1}(\infty) = \{p_i \mid x_i < 0\}$, and $f$ is required to ramify at $p_i$ with index $|x_i|$. An isomorphism between two such maps is a pair of isomorphisms between sources and targets, all fitting into a commutative square.

Put $\calM(\bfx)$ for the moduli space of all such maps. When $n\geq 3$, we have an isomorphism $\calM(\bfx) \cong \calM_{0, n}$, where $\calM_{0,n}$ is the moduli space of smooth $n$-marked curves of genus $0$. The space $\calM(\bfx)$ admits a compactification $\Mbar(\bfx)$ which parameterizes maps where the source is a tree of $\bbP^1$'s and the target is a chain of $\bbP^1$'s; see Section \ref{background} for a precise definition. For each choice of $\bfx$, the space $\Mbar(\bfx)$ is birational to the Deligne--Mumford--Knudsen moduli space $\Mbar_{0, n}$, via the map that remembers only the stabilized source curve. In this paper, we study how topological invariants of $\Mbar(\bfx)$ depend on the input datum $\bfx$.

Our first main theorem gives a recursive algorithm for computing the topological Euler characteristic of the coarse moduli space $\Mbar(\bfx)$ when
\[\bfx = (n, \underbrace{-1, \ldots, -1}_{n \text{ times}}), \]
for $n \geq 2$. Put $\Mbar_n$ for the resulting moduli space, which admits a birational morphism \[\pi_n : \Mbar_n \to \Mbar_{0,n+1}.\] Let $\Mbar_{0,n+1}(k) \subseteq \Mbar_{0,n+1}$
denote the locus of stable pointed curves with exactly $k$ irreducible components. We set \[\Mbar_n(k) \defeq \pi_n^{-1}(\Mbar_{0, n+ 1}(k)) \subseteq \Mbar_n,\] and define the bivariate generating function
\[\Psi(s, t) \defeq \sum_{k \geq 1} \sum_{n \geq 2} \chi(\Mbar_n(k)) \frac{s^k t^n}{k! n!},  \]
where $\chi$ denotes the topological Euler characteristic.
\begin{customthm}{A}\label{Recursion}
The generating function $\Psi$ satisfies the differential equation
\[\frac{\partial \Psi}{\partial s} = (1 +t)\left(\log(1 + t) + \exp\left( -\frac{\Psi}{1 + t} \right)\right) + \Psi(\log(1 + t) + 1) - 2t - 1, \]
with initial condition $\Psi(0, t) = 0$.
\end{customthm}

Theorem \ref{Recursion} allows for the recursive calculation of the numbers $\chi(\Mbar_n(k))$ and hence the numbers $\chi(\Mbar_n)$, since $\chi(\Mbar_n) = \sum_{k = 1}^{n - 1} \chi(\Mbar_n(k))$. Table \ref{EulerChars} compares the Euler characteristics of the birational spaces $\Mbar_n$ and $\Mbar_{0, n + 1}$ for $n$ up to 19. In both cases, the Euler characteristic is equal to the dimension of the rational cohomology ring.

Our second main theorem states that when viewed as a function of $\bfx$, the class of $\Mbar(\bfx)$ in the Grothendieck ring of varieties is constant on the chambers of a well-studied decomposition of $\Z^n$. Set $[n] \defeq \{1, \ldots, n\}$, and for a subset $I \subseteq [n]$, put
\[W_I \defeq  \left\{(x_1, \ldots, x_n) \mid \sum_{i \in I}x_i = 0\right\} \subset \R^n, \]
to define 
\begin{equation}\label{admissibledata}
    A_n \defeq W_{[n]} \smallsetminus \left(\bigcup_{\substack{{I \subset [n]} \\ I \neq \varnothing, [n]}} W_I \right).
\end{equation}
Throughout this paper, we will assume that our ramification datum $\bfx$ is selected from $A_n \cap \Z^n$. The walls $W_I$ for $I \subsetneq [n]$ define a hyperplane arrangement in $\R^n$, known as the \textit{\textbf{resonance arrangement}}. They induce a decomposition of $A_n$ into connected components which are called \textit{\textbf{resonance chambers}}. These terms were coined by Shadrin--Shapiro--Vainshtein ~\cite{SSV} in their study of Hurwitz numbers. We prove that the class $[\Mbar(\bfx)]$ is constant on these chambers.
\begin{customthm}{B}\label{LocallyConstant}
Let $[\Mbar(\bfx)] \in K_0(\mathsf{Var}/\C)$ denote the class of $\Mbar(\bfx)$ in the Grothendieck ring of varieties. Then, as a function of $\bfx \in A_n \cap \Z^n$, the class $[\Mbar(\bfx)]$ is constant on resonance chambers.
\end{customthm}
Theorem \ref{LocallyConstant} implies that Theorem \ref{Recursion} extends to the the calculation of $\chi(\Mbar(\bfx))$ for any $\bfx$ chosen from the \textbf{\textit{central chamber}}, which we define to be the one containing $(n, -1, \ldots, -1)$.

\begin{table}
\begin{tabular}{|l|l|l|}
\hline
\rule{0pt}{2.5ex}$n$  & $\chi(\Mbar_n)$ & $\chi(\Mbar_{0,n+1})$\\ \hline
2  & 1 & 1                                                      \\ \hline
3  & 2 & 2                                                      \\ \hline
4  & 10 & 7                                                     \\ \hline
5  & 84 & 34                                                    \\ \hline
6  & 1108 & 213                                                 \\ \hline
7  & 20824 & 1630                                               \\ \hline
8  & 530528 & 14747                                             \\ \hline
9  & 17578464 & 153946                                          \\ \hline
10 & 734772384 & 1821473                                        \\ \hline
11 & 37814132256 & 24087590                                     \\ \hline
12 & 2349344349504 & 352080111                                  \\ \hline
13 & 173367352211520 & 5636451794                               \\ \hline
14 & 14989230432337536 & 98081813581                            \\ \hline
15 & 1500796146336385152 & 1843315388078                        \\ \hline
16 & 172277450643084049920 & 37209072076483                     \\ \hline
17 & 22474724472542045216256 & 802906142007946                  \\ \hline
18 & 3306538057482623252067840 & 18443166021077145              \\ \hline
19 & 544879611875655894561850368 & 449326835001457846           \\ \hline
\end{tabular}
\caption{Euler characteristics of $\Mbar_n$, computed using Theorem \ref{Recursion}, compared with the corresponding values for $\Mbar_{0,n+1}$.}
\label{EulerChars}
\end{table}

\subsection{Related work}
The spaces $\Mbar(\bfx)$ and higher genus counterparts have attracted much interest over the past few decades, mainly for their role in enumerative geometry. Analogues of this moduli space where the source has arbitrary genus were studied by Faber--Pandharipande ~\cite{FaberPandharipande}, who showed that Gromov--Witten classes on these spaces have tautological pushforwards to the moduli spaces of curves. These spaces were also studied in the context of relative virtual localization by Graber--Vakil ~\cite{GraberVakil}, also with applications to tautological rings. Janda \textit{et al.} ~\cite{JPPZ} proved explicit formulas for the pushforwards of the virtual fundamental classes of $\Mbar(\bfx)$ and higher genus analogues in the tautological rings. Ranganathan ~\cite{Rang17} studied the generalization of $\Mbar(\bfx)$ where the target is replaced by an arbitrary toric variety, relative to its toric boundary.

The spaces $\Mbar(\bfx)$ are also central to Hurwitz theory, which is concerned with the enumeration of branched covers of $\bbP^1$ with fixed ramification profiles. Certain Hurwitz numbers can be expressed as integrals over spaces of relative stable maps, as is noted by Goulden--Jackson--Vakil ~\cite{GJV}, who also showed that as a function of the ramification datum $\bfx$, the double Hurwitz numbers are piecewise polynomial. Shadrin--Shapiro--Vainshtein ~\cite{SSV} proved that in genus zero, the chambers of polynomiality are given by the resonance chambers. This result has been extended to higher genus by Cavalieri--Johnson--Markwig ~\cite{ChamberStructure}.

The space $\Mbar(\bfx)$ has also been studied in the context of tropical geometry by Cavalieri--Markwig--Ranganathan ~\cite{CMR}. They showed how to construct $\Mbar(\bfx)$ as a tropical compactification of $\calM(\bfx)$ inside a toric variety, and consequently obtained an equality of classical and tropical Gromov--Witten invariants of $\bbP^1$. Many of the combinatorial techniques in this paper are inspired by their work, as well as work of Cavalieri--Johnson--Markwig ~\cite{CJM} on tropical Hurwitz theory.

Our proof of Theorem \ref{Recursion} is in the spirit of enumerative combinatorics, as the generating function $\Psi$ can be expressed as a sum over all rooted stable trees. The strategy of proof is inspired by previous work of Manin ~\cite{Manin} on the generating function for Euler characteristics of the moduli spaces $\Mbar_{0, n}$. Operadic versions of this technique have been used to calculate topological invariants of $\Mbar_{0,n}$ by Getzler ~\cite{Getzler} and Kontsevich spaces of stable maps by Getzler--Pandharipande ~\cite{GetzlerPandharipande}.

\subsection{Outline of the paper}
In Section \ref{background} we give a precise definition of the spaces $\Mbar(\bfx)$, recall their stratification by combinatorial type, and prove Theorem \ref{LocallyConstant}. In Section \ref{TreeSums}, we prove Theorem \ref{Recursion} by expressing the relevant generating function as a sum over trees. In Section \ref{NewHorizons}, we close the paper with a discussion of potential next steps in the study of the spaces $\Mbar(\bfx)$, and make two conjectures: Conjecture \ref{complexity} concerns the asymptotic behavior of the Euler characteristic in the central chamber, while Conjecture \ref{ChernConjecture} posits a chamber structure for the Chern numbers of $\Mbar(\bfx)$.

\subsection*{Acknowledgments}
I am grateful to Dhruv Ranganathan for suggesting this line of research, many helpful discussions, and careful feedback on an early draft of this article. I also thank Navid Nabijou for useful conversations. This paper was completed during a visit to the DPMMS at the University of Cambridge, and I thank them for providing ideal working conditions. This work was supported by an NSF Graduate Research Fellowship.

\section{The moduli space \texorpdfstring{$\Mbar(\bfx)$}{Mx} and its stratification}\label{background}

Fix a vector of integers $\bfx = (x_1, \ldots, x_n) \in A_n \cap \Z^n$, where $A_n$ is as defined in (\ref{admissibledata}). Recall that the \textit{\textbf{dual tree}} of a nodal curve $X$ of arithmetic genus zero is formed by giving a vertex for each irreducible component of $X$, and connecting two such vertices by an edge when the corresponding components are joined by a node in $X$. Labelled vertices of valence one are then added to indicate the distribution of the marked points among the irreducible components of $X$.
\begin{defn}
A \textit{\textbf{genus zero rubber map}} is a morphism \[f:(X, p_1, \ldots, p_n) \to (Y, 0_Y, \infty_Y),\]
where 
\begin{itemize}
    \item $X$ is a nodal curve of arithmetic genus $0$ with $p_i \in X$ smooth;
    \item $Y$ is a nodal curve of arithmetic genus $0$ whose dual tree is a path;
    \item the points $0_Y, \infty_Y \in Y$ are smooth, and lie on the two extreme components of $Y$;
    \item the preimage of each node of $Y$ is a union of nodes of $X$;
    \item if we lift $f$ to a morphism between the normalizations of source and target, the ramification indices at the two preimages of each node of $X$ agree;
    \item $f^{-1}(0) = \{p_i \mid x_i > 0 \}$ and $f^{-1}(\infty) = \{ p_i \mid x_i < 0\}$, and $f$ ramifies at $p_i$ with ramification index $|x_i|$.
\end{itemize}
\end{defn}
An isomorphism between two such maps
\[f:(X, p_1, \ldots, p_n) \to (Y, 0_Y, \infty_Y) \mbox{ and } f':(X', p_1', \ldots, p_n') \to (Y', 0_{Y'}, \infty_{Y'}) \]
is a pair of isomorphisms $\varphi$ and $\psi$ which fit into a commuting square:
\begin{equation}\label{Square}
\begin{tikzcd}
&(X, p_1, \ldots, p_n)\arrow[r, "f"] \arrow[d, "\varphi"] &(Y, 0_{Y}, \infty_{Y})\arrow[d, "\psi"]\\
&(X', p_1', \ldots, p_n')\arrow[r, "f'"] &(Y', 0_{Y'}, \infty_{Y'})
\end{tikzcd}.
\end{equation}
The term ``rubber" in the above definition refers to the fact that the target is unparameterized. For example, when $Y = \bbP^1$, two maps that differ by the $\C^*$-action are considered equivalent. A genus zero rubber map is \textit{\textbf{stable}} if it has finitely many automorphisms. Given a genus zero rubber map $f:(X, p_1, \ldots, p_n) \to (Y, 0_Y, \infty_Y)$, we say an irreducible component $X'$ of $X$ is a \textit{\textbf{trivial bubble}} if $X'$ is mapped to its image with full ramification over $0$ and $\infty$, and no ramification elsewhere, i.e. the map is of the form $[z_0:z_1] \mapsto [z_0^d: z_1^d]$ when restricted to $X'$. The following proposition gives a concrete interpretation of stability in terms of trivial bubbles.
\begin{prop}
A genus zero rubber map $f:(X, p_1, \ldots, p_n) \to (Y, 0_Y, \infty_Y)$ is stable if and only if
\begin{enumerate}[(1)]
    \item each component of $X$ which is contracted by $f$ has at least three special points, where a special point is a marked point or a node, and
    \item for each component $Y'$ of $Y$, there is a component of $X$ mapping to $Y'$ which is not a trivial bubble.
\end{enumerate}
\end{prop}

The described moduli problem is representable by a smooth Deligne-Mumford stack. Unless otherwise stated, we will work with the coarse moduli scheme, which we denote by $\Mbar(\bfx)$. This is a projective scheme with finite quotient singularities, and the inclusion $\calM(\bfx) \hookrightarrow \Mbar(\bfx)$ is a toroidal embedding. The dual complex of the boundary is a subdivision of the moduli space $\Delta_{0,n}$ of $n$-marked genus zero tropical curves, and is studied in ~\cite{CMR}.
\subsection{Combinatorial types of rubber stable maps}
Each map in $\Mbar(\bfx)$ can be discretized into a morphism of decorated trees. We will now axiomatize the maps of trees which arise in this way, and consequently obtain a stratification of $\Mbar(\bfx)$, similar to the dual graph stratifications of moduli spaces of curves. Given a finite tree $T$, we write $V(T)$ and $E(T)$ for the set of vertices and edges of $T$, respectively. A vertex of valence one in $T$ is called a \textit{\textbf{leaf}} of $T$; the set of leaves is denoted by $L(T)$. An edge of $T$ which is connected to a leaf will be called an \textit{\textbf{end}}. Vertices which are not leaves are called \textit{\textbf{internal vertices}}, and we put $I(T)$ for the set of all internal vertices.
\begin{defn}
Let $S$ be a finite set. An \textit{\textbf{$S$-marked tree}} is a pair $\T = (T, m)$ where $T$ is a finite connected tree and $m: S \to L(T)$ is a bijection. An \textit{\textbf{isomorphism}} of $S$-marked trees $\phi: \T \to \T'$ is an isomorphism of the underlying trees that respects the marking functions. An $S$-marked tree is said to be \textit{\textbf{stable}} if it has no vertices of valence two.
\end{defn}
For an integer $n > 0$, we say $\T$ is an \textbf{\textit{$n$-marked tree}} if the marking set is $S = \{1, \ldots, n\}$. We put $\Gamma_{0,n}$ for the set of isomorphism classes of all stable $n$-marked trees.
\begin{defn}
An \textit{\textbf{$n$-marked combinatorial rubber map}} is a map $f: \T \to \P$, where:
\begin{itemize}
    \item $\T$ is an $n$-marked tree
    \item $\P$ is a $\{0, \infty\}$-marked tree whose underlying tree is a path, and
    \item $f$ maps leaves to leaves, internal vertices to internal vertices, and edges to edges.
\end{itemize}
An isomorphism between two such maps $f: \T \to \P$ and $f': \T' \to \P'$ is a pair of isomorphisms fitting into a commuting square as in (\ref{Square}).
\end{defn}
We say an $n$-marked combinatorial rubber map $f: \T \to \P$ is \textit{\textbf{stable}} if each internal vertex in $\P$ has at least one preimage in $\T$ which has valence greater than two. Note that a combinatorial rubber map may be stable even though the source tree is not. However, the source of a combinatorial rubber map always has a \textit{\textbf{stabilization}} which lies in $\Gamma_{0,n}$, obtained by smoothing bivalent vertices. 

Given a relative stable map $f: (X, p_1, \ldots, p_n) \to (Y, 0_Y, \infty_Y)$, we put $\Gamma(X)$ for the $n$-marked dual tree of $X$, and $\Gamma(Y)$ for the $\{0, \infty\}$-marked dual tree of $Y$. The morphism $f: (X, p_1, \ldots, p_n) \to (Y, 0_Y, \infty_Y)$ induces an $n$-marked combinatorial rubber stable map \[\pi: \Gamma(X) \to \Gamma(Y),\] since nodes of $X$ map to nodes of $Y$ and marked points map to marked points. We refer to the combinatorial rubber stable map $\pi: \Gamma(X) \to \Gamma(Y)$ as the \textit{\textbf{combinatorial type}} of the stable map $f$. See Figure \ref{ComboTypeExampleFig}.

\begin{figure}[h]
    \centering
    \includegraphics[scale=1.25]{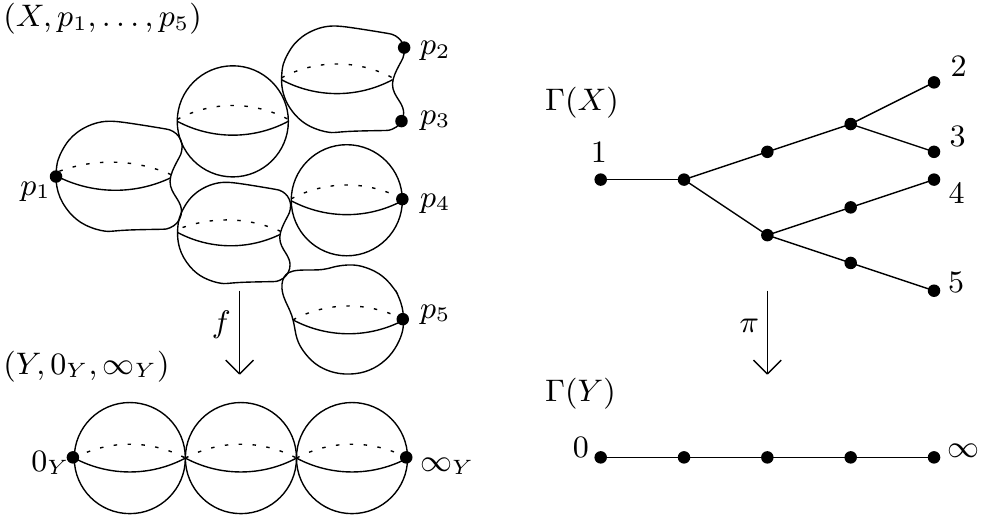}
    \caption{A rubber stable map $f$ and its combinatorial type.}
    \label{ComboTypeExampleFig}
\end{figure}

\subsection{Combinatorial types, partial orderings, and weights}

Given a combinatorial type $\pi: \T \to \P$, we can put a total ordering on $V(\P) \coprod E(\P)$ using the path structure of $\P$. We agree that $0 \in L(\P)$ is the minimal element in this ordering, and for $\alpha, \beta \in V(\P) \coprod E(\P)$, we have $\alpha < \beta$ if and only if $\alpha$ is nearer $0$ than $\beta$ in $\P$. In this case we will say that $\alpha$ is to the left of $\beta$, and we call this the \textit{\textbf{left-to-right}} ordering of $\P$. Using the map $\pi$, we can pull back the left-to-right ordering to get a \textit{partial} ordering of $V(\T) \coprod E(\T)$, where two elements are incomparable if and only if they are both in the same fiber of $\pi$. Given a vertex $v \in V(\T)$, we define sets \[\mathrm{Left}(v), \mathrm{Right}(v) \subset E(\T)\] as follows: the elements of $\mathrm{Left}(v)$, respectively $\mathrm{Right}(v)$, are those edges $e$ containing $v$ such that $\pi(e)$ is strictly less, respectively strictly greater, than $\pi(v)$ in the total ordering on $V(\P) \coprod E(\P)$. 

As a stable map \[f: (X, p_1, \ldots, p_n) \to (Y, 0_Y, \infty_Y)\] has a well-defined ramification index at each node of $X$, there is a natural weight function
\begin{equation}\label{WeightFunction}
w : E(\Gamma(X)) \to \Z_{>0}
\end{equation}
assigning to an end the ramification index at the corresponding marked point, and assigning to an internal edge the ramification index at the corresponding node. By basic considerations about morphisms of smooth curves, we must have
\begin{equation}\label{Balancing}
    \sum_{e \in \mathrm{Left(v)}} w(e) = \sum_{e \in \mathrm{Right(v)}} w(e)
\end{equation}
for any $v \in V(\T)$, assuming that $\pi: \T \to \P$ is a combinatorial type of a relative stable map, with the weight function given by ramification indices. If $\T = (T, m)$ is the source of a combinatorial rubber map with weight function $w$, and $T_1 \subset T$ is a subtree, we put
\[w(T_1) \defeq \sum_{i \in L(T_1) \cap L(T)} x_i.\]

Note that while the weight function for edges (\ref{WeightFunction}) is always positive, the weight function for subtrees can be either positive or negative. When $T_1$ is a connected component of the tree obtained from $T$ by deleting a single edge, the following result expresses $w(T_1)$ in terms of the contribution of edges connected to a particular vertex.

\begin{lem}\label{LocalCalculation}
Suppose $\pi: \T \to \P$ is the combinatorial type of a relative stable map, with weight function $w: E(\T) \to \Z_{>0}$ given by the ramification indices. Let $v \in I(\T)$ be an internal vertex, and let $e' \in E(\T)$ be an internal edge containing $v$. Let $T_1$ be the connected component of $\T \smallsetminus \{e'\}$ which contains $v$. Then we have
\[w(T_1) = \sum_{e \in \mathrm{Left}(v) \cap E(T_1)} w(e) -  \sum_{e \in \mathrm{Right}(v) \cap E(T_1)} w(e). \]
\end{lem}
\begin{proof}
Suppose that $\P$ has $r$ internal edges, labelled from left-to-right as $f_1, \ldots, f_r$, with $0$-end labelled as $f_0$ and $\infty$-end labelled as $f_{r + 1}$. Then there exists some $k$ such that $v \in \pi^{-1}(u_k)$, where $u_k$ is the vertex where $f_k$ and $f_{k + 1}$ meet. By definition, we have
\[w(T_1) = \sum_{e \in \pi^{-1}(f_0) \cap E(T_1)} w(e) - \sum_{e \in \pi^{-1}(f_{r + 1}) \cap E(T_1)}w(e). \]
Now we claim that for any $i$ with $0 \leq i \leq r$ and $i \neq k$, we have
\begin{equation}\label{CollapsingToCenter}
\sum_{e \in \pi^{-1}(f_i) \cap E(T_1)} w(e) = \sum_{e \in \pi^{-1}(f_{i + 1}) \cap E(T_1)}w(e).
\end{equation}
Indeed, for any vertex $s \in V(T_1)$ with $s \neq v$, we must have $\mathrm{Left}(s), \mathrm{Right}(s) \subseteq E(T_1)$: since $s$ and $v$ are connected by a path in $T_1$, it must be that any edge containing $s$ is in the same connected component as $v$. Therefore, if we let $u_i$ be the vertex where $f_i$ and $f_{i + 1}$ meet in $\P$, we have
\[\sum_{e \in \pi^{-1}(f_i) \cap E(T_1)} w(e) = \sum_{s \in \pi^{-1}(u_i) \cap V(T_1)} \left(\sum_{e \in \mathrm{Left}(s)} w(e)\right) \]
and 
\[\sum_{e \in \pi^{-1}(f_{i + 1}) \cap E(T_1)} w(e) = \sum_{s \in \pi^{-1}(u_i) \cap V(T_1)} \left(\sum_{e \in \mathrm{Right}(s)} w(e)\right), \]
so (\ref{CollapsingToCenter}) follows from (\ref{Balancing}). Applying (\ref{CollapsingToCenter}) for $i = 0, \ldots, k-1$, we see that
\[\sum_{e \in \pi^{-1}(f_0) \cap E(T_1)} w(e) = \sum_{e \in \pi^{-1}(f_k) \cap E(T_1)} w(e). \]
Similarly, applying (\ref{CollapsingToCenter}) for $i = k+1, \ldots, r$, we get
\[\sum_{e \in \pi^{-1}(f_{r + 1}) \cap E(T_1)}w(e) = \sum_{e \in \pi^{-1}(f_{k + 1}) \cap E(T_1)}w(e), \]
so that
\[w(T_1) = \sum_{e \in \pi^{-1}(f_k) \cap E(T_1)} w(e) - \sum_{e \in \pi^{-1}(f_{k + 1}) \cap E(T_1)}w(e). \]
If we set $\pi^{-1}(u_k) \cap V(T_1) = \{v, j_1, \ldots, j_\ell\}$, then we have $\mathrm{Left}(j_i), \mathrm{Right}(j_i) \subset E(T_1)$ for all $i$. Thus, using the above expression and (\ref{Balancing}), we have
\begin{align*}
    w(T_1) &= \sum_{e \in \mathrm{Left}(v) \cap E(T_1)} w(e) - \sum_{e \in \mathrm{Right}(v) \cap E(T_1)} w(e) +  \sum_{i = 1}^{\ell}\left(\sum_{e \in \mathrm{Left}(j_i)}  w(e)- \sum_{e \in \mathrm{Right}(j_i)}w(e)\right) \\&= \sum_{e \in \mathrm{Left}(v) \cap E(T_1)} w(e) - \sum_{e \in \mathrm{Right}(v) \cap E(T_1)} w(e),
\end{align*}
as we wanted to show.
\end{proof}
We now work towards the enumeration of boundary strata of $\Mbar(\bfx)$ in terms of pairs consisting of a tree $\T \in \Gamma_{0,n}$ together with an ordered partition of $V(\T)$, satisfying certain technical properties. These partitions will be defined in terms of the following directed tree structure on $\T$.
\begin{defn}
Given $\bfx \in A_n$, each tree $\T \in \Gamma_{0,n}$ can be made into a directed tree as described below.
\begin{enumerate}
    \item The edge of $\T$ incident to leaf $i$ is directed inwards if $x_i > 0$, and outwards if $x_i < 0$.
    \item Upon deletion, each internal edge $e$ of $\T = (T, m)$ disconnects $T$ into two components $T_1$ and $T_2$. As $\bfx\in A_n$, we must have $w(T_1) + w(T_2) = 0$, and $w(T_1), w(T_2) \neq 0$, hence exactly one of the values $w(T_1), w(T_2)$ is negative, and the other is positive. The edge $e$ is oriented away from the positive component, and towards the negative component.
\end{enumerate}
This directed tree structure on $\T$ will be called the \textit{\textbf{$\bfx$-directing of $\T$}}.
\end{defn}
See Figure \ref{xdirectingexamplefig} for two different $\bfx$-directings of the same tree $\T \in \Gamma_{0,6}$, corresponding to two different choices of $\bfx$.
\begin{figure}[h]
    \centering
    \includegraphics[scale=1.25]{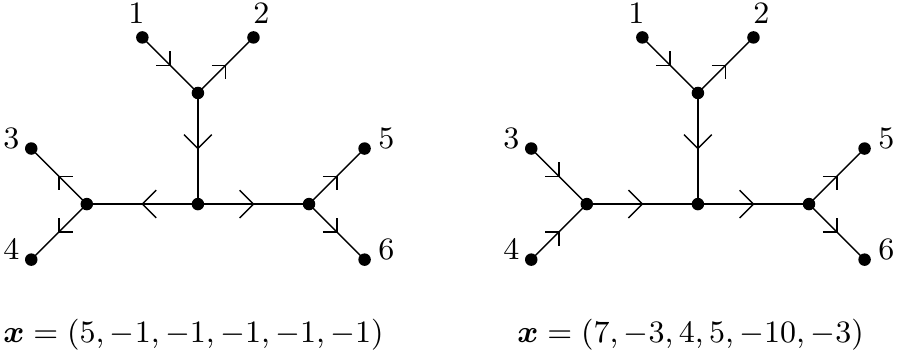}
    \caption{The $\bfx$-directings of a tree $\T \in \Gamma_{0,6}$ for two different choices of $\bfx$.}
    \label{xdirectingexamplefig}
\end{figure}

The following lemma follows immediately from the definition of an $\bfx$-directing.
\begin{lem}\label{DirectingLocallyConstant}
For fixed $\T \in \Gamma_{0,n}$, the $\bfx$-directing of $\T$ depends only on the resonance chamber containing $\bfx$.
\end{lem}

Once a tree $\T \in \Gamma_{0,n}$ is $\bfx$-directed, we get a partial ordering of the vertices of $\T$: $v \leq v'$ if and only if $v = v'$ or there is a directed path from $v$ to $v'$ in $\T$. We will denote this partial ordering by $\leq_{\bfx}$. Strata of $\Mbar(\bfx)$ corresponding to maps whose source curve has stabilized dual graph equal to $\T$ will be seen to be in one-to-one correspondence with ordered partitions of $V(\T)$ which are compatible with this partial ordering.
\begin{defn}\label{admissiblepartitions}
Suppose $S$ is a finite set equipped with a partial order $\preccurlyeq$. Then a \textit{\textbf{$\preccurlyeq$-admissible ordered partition}} of $S$ is an ordered partition $S = P_0 \coprod \cdots \coprod P_{r + 1}$ into nonempty blocks, such that
\begin{enumerate}[(1)]
    \item any pair of distinct elements in block $P_i$ are incomparable under $\preccurlyeq$, and
    \item if $a \in P_i$ and $b \in P_j$ with $i < j$, then either $a \preccurlyeq b$ or $a$ and $b$ are incomparable under $\preccurlyeq$.
    \item if $a \in P_i$ with $1 \leq i \leq r$, then there exist indices $j, k$ with $i < j$, $i > k$, and elements $b \in P_j$, $c \in P_k$, such that $a \preccurlyeq b$ and $c \preccurlyeq a$.
\end{enumerate}
\end{defn}

Let $\T \in \Gamma_{0,n}$ and $\bfx \in A_n \cap \Z^n$. We set
\[L_{\bfx}^+(\T) = \{ i \in L(\T) \mid x_i > 0 \} \mbox{ and } L_{\bfx}^-(\T) = \{ i \in L(\T) \mid x_i < 0 \}. \]
Finally, we set $\calP_{\bfx}(\T)$ to be the set of $\leq_{\bfx}$-admissible ordered partitions of $V(\T)$, such that the first block of the partition is given by $L_{\bfx}^+(\T)$ and the last block is given by $L_{\bfx}^-(\T)$.

The following lemma enumerates combinatorial types of maps in $\Mbar(\bfx)$ in terms of the set $\calP_{\bfx}(\T)$.
\begin{lem}\label{Orderings}
Fix $\bfx \in A_n \cap \Z^n$ and $\T \in \Gamma_{0, n}$. Let $\calC_{\bfx}(\T)$ denote the set of isomorphism classes of combinatorial types of maps in $\Mbar(\bfx)$ whose stabilized source graph is equal to $\T$. Then there is a bijection \[\calC_{\bfx}(\T) \to \calP_{\bfx}(\T). \]
\end{lem}
\begin{proof} 
We will explicitly describe a bijection $\calC_{\bfx}(\T) \to \calP_{\bfx}(\T)$. Suppose given a combinatorial type $\hat{\T} \to \P$, such that $\hat{\T}$ stabilizes to $\T$. Suppose further that $\P$ has $r + 2$ vertices for some $r > 0$. We set $v_0 \in V(\P)$ for the leaf labelled by $0$, and $v_{r + 1}$ for the leaf labelled by $\infty$. We then label the remaining vertices $v_i \in V(\P)$ using the left-to-right ordering, so e.g. $v_1$ is the unique non-leaf vertex connected to $v_0$. Then we get an ordered partition of $V(\T)$ by setting $P_i = \pi^{-1}(v_i)$ for $0 \leq i \leq r+1$; here we identify $V(\T)$ with those elements of $V(\hat{\T})$ which are not bivalent. Each $P_i$ is nonempty by the stability condition, and $P_0 = L^{+}_{\bfx}(\T)$ while $P_{r + 1} = L^{-}_{\bfx}(\T)$. We must now check that this partition satisfies all three conditions of Definition \ref{admissiblepartitions}. We first check condition (2). Suppose $i < j$ with $u_1 \in P_i$ and $u_2 \in P_j$, and that $u_1$ and $u_2$ are comparable under $\leq_{\bfx}$; we must show that $u_1 \leq_{\bfx} u_2$. Suppose for sake of contradiction that $u_2 \leq_{\bfx} u_1$. This means that in $\T$, there is a directed path $J$ from $u_2$ to $u_1$. Let $\hat{J}$ be the corresponding path in $\hat{\T}$. We claim that there must be $e'$ in $\hat{J}$ such that $e' \in \mathrm{Right}(u_1)$. Indeed, the geometric realizations of $\hat{\T}$ and $\P$ are finite connected 1-dimensional CW-complexes, and $\pi$ is a cellular, hence continuous, map. Therefore, $\pi$ must take the path $\hat{J}$ to a path between $\pi(u_2)$ and $\pi(u_1)$ in $\P$, from which it follows that $\hat{J}$ contains such an edge $e'$. By our assumption on $\hat{J}$ being directed from $u_2$ to $u_1$, we must have $w(T_1) < 0$, where $T_1$ is the connected component of $\hat{\T} \smallsetminus \{e'\}$ which contains $u_1$. By Lemma \ref{LocalCalculation}, this implies that
\[\sum_{e \in \mathrm{Left}(u_1)} w(e) - \sum_{\substack{{e \in \mathrm{Right}(u_1)}\\{e \neq e'}}} w(e) < 0. \]
However, we have
\[\sum_{e \in \mathrm{Left}(u_1)} w(e) - \sum_{\substack{{e \in \mathrm{Right}(u_1)}\\{e \neq e'}}} w(e) > \sum_{e \in \mathrm{Left}(u_1)} w(e) - \sum_{e \in \mathrm{Right}(u_1)} w(e),  \]
since all the weights are strictly positive, and
\[\sum_{e \in \mathrm{Left}(u_1)} w(e) - \sum_{e \in \mathrm{Right}(u_1)} w(e) = 0\]
by (\ref{Balancing}). We have thus arrived at a contradiction, so it must be that if $u_1$ and $u_2$ are comparable under $\leq_{\bfx}$, then $u_1 \leq_{\bfx} u_2$, and the ordered partition satisfies condition (2). To check condition (1), note that no two elements of $P_i$ are connected by an edge in $\T$, so any path between $u_1, u_2 \in P_i$ must contain an element $u_3 \in P_j$ for $j \neq i$. If $u_3$ cannot be chosen to be comparable with both $u_1$ and $u_2$, then it must be that $u_1$ and $u_2$ are incomparable, since no path between them is directed. Otherwise, suppose that $u_3$ is comparable with both $u_1$ and $u_2$. Then if $j > i$, by condition (2), we have $u_1, u_2 \leq_{\bfx} u_3$, and if $j < i$, then $u_3 \leq_{\bfx} u_1, u_2$. In any case, the path between $u_1$ and $u_2$ cannot be directed, so $u_1$ and $u_2$ are incomparable under $\leq_{\bfx}$, and condition (1) holds. Finally, since $\hat{\T}$ satisfies (\ref{Balancing}), we have that $\mathrm{Left}(v)$ and $\mathrm{Right}(v)$ are nonempty for all $v \in I(\hat{\T})$, so condition (3) is satisfied. In particular, we get a map $\calC_{\bfx}(\T) \to \calP_{\bfx}(\T)$ as described.

We will now construct an inverse to this map. Suppose given a $\leq_{\bfx}$-admissible ordered partition
\[V(\T) = P_0 \coprod \cdots \coprod P_{r + 1}, \]
where $P_0 = L^+_{\bfx}(\T)$ and $P_{r + 1} = L^{-}_{\bfx}(\T)$. 
Then we construct a target path $\P$ with $r + 2$ vertices $v_0, v_1, \ldots, v_r, v_{r + 1}$ so that $v_0$ is the leaf labelled by $0$ and $v_{r + 1}$ is the leaf labelled by $\infty$. We construct our source tree $\hat{\T}$ by subdividing edges of $\T$ with bivalent vertices: if $j > i$ and $e$ is an edge between $u_i \in P_i$ and $u_j \in P_j$, we subdivide $e$ by introducing $j - i - 1$ bivalent vertices. To construct the morphism $\pi$, we map the vertices of $P_i$ to $v_i$, and a path between $u_i \in P_i$ and $u_j \in P_j$ is mapped to the unique path between $v_i$ and $v_j$, which has length $j - i$. Since each $P_i$ is nonempty, each vertex in $\P$ has at least one preimage under $\pi$ with valence at least three, so the resulting morphism of trees is stable. Since the given partition satisfies condition (3), we have that $\mathrm{Left}(v)$ and $\mathrm{Right}(v)$ are nonempty for all $v \in I(\hat{\T})$. This means that if we put a directed structure on $\hat{\T}$ by directing $e$ towards $v$ whenever $e \in \mathrm{Left}(v)$, the resulting directed tree has no sources or sinks. Therefore there is a unique weight function $w: E(\hat{\T}) \to \Z_{> 0}$ such that $w$ is equal to $|x_i|$ on the $i$th end of $\hat{\T}$ and $w$ satisfies the balancing condition (\ref{Balancing}). This is proven in the case where $\hat{\T}$ is trivalent and the weight function is signed in ~\cite[Lemma 6.4]{CJM}; the proof therein carries over with no significant changes. It is now straightforward to construct a relative stable map with combinatorial type $\pi: \hat{\T} \to \P$ such that the weight function $w$ comes from the ramification data of the stable map. As such, we have described an inverse map $\calP_{\bfx}(\T) \to \calC_{\bfx}(\T)$, and the proof is complete.
\end{proof}

For a fixed $n$-marked combinatorial rubber stable map $\pi: \T \to \P$, we put \[\calM_{\bfx}(\pi: \T \to \P) \subset \Mbar(\bfx)\]
for the locally closed stratum of maps having combinatorial type equal to $\pi: \T \to \P$. Then, if this stratum is nonempty, one has an isomorphism of stacks
\begin{equation}\label{StrataFactoring}
\calM_{\bfx}(\pi: \T \to \P) \cong \prod_{\substack{v \in V(\T) \\ \mathrm{val}(v) \geq 3}} \calM_{0, \val(v)}  \times \prod_{\substack{v \in V(\T) \\ \mathrm{val}(v) = 2}} B\mu_{w(v)} \times \prod_{v \in V(\P)} (\C^*)^{|\{u \in \pi^{-1}(v) \mid \val(u) \geq 3\}| - 1}.
\end{equation}
In the above, for a valence two vertex $v$, we have put $w(v)$ for the weight of either edge adjacent to $v$; this is well-defined by (\ref{Balancing}). The $\C^*$ and $\calM_{0, \val(v)}$ factors in (\ref{StrataFactoring}) appear because a map $\bbP^1 \to \bbP^1$ is determined, up to the action of $\C^*$ on the target, by a divisor of degree zero. The factors of $B\mu_{w(v)}$ account for the action of roots of unity on trivial bubbles, corresponding to the degree two vertices. Given a stable tree $\T \in \Gamma_{0,n}$, we put $\calM_{\bfx}(\T) \subset \Mbar(\bfx)$ for the stratum of maps whose source curves have dual graphs isomorphic to $\T$, upon stabilizing.
\begin{prop}\label{TreeStrataProp}
Let $\T \in \Gamma_{0, n}$. We have an equality
\[[\calM_{\bfx}(\T)] = \prod_{v \in I(\T)} [\calM_{0, \val(v)}] \sum_{\scrP \in \mathcal{P}_{\bfx}(\T)} [\C^{*}]^{|I(\T)| - \ell(\scrP)+ 2} \]
in $K_0(\mathsf{Var}/\C)$, where for an ordered partition $\mathscr{P} \in \calP_{\bfx}(\T)$, we put $\ell(\scrP)$ for the number of blocks of $\scrP$.
\end{prop}
\begin{proof}
Let $\pi: \hat{\T} \to \P$ be a combinatorial type such that $\hat{\T}$ stabilizes to $\T$. By (\ref{StrataFactoring}), we see that upon taking coarse moduli spaces, bivalent vertices do not contribute to the class of $\M_{\bfx}(\pi: \hat{\T} \to \P)$ in $K_0(\mathsf{Var}/\C)$. Let $\mathrm{st}: V(\T) \to V(\hat{\T})$ be the natural inclusion of the vertex set of $\T$ into that of $\hat{\T}$, so that the image of $\mathrm{st}$ is precisely those vertices of $\hat{\T}$ which are not bivalent. Define $\pi_{\mathrm{st}} : V(\T) \to V(\P)$ by $\pi_{\mathrm{st}} =  \pi \circ \mathrm{st}$. Then we have
\begin{align*}
[\M_{\bfx}(\pi: \hat{\T} \to \P)] &= \left(\prod_{v \in I(\T)} [\calM_{0, \val(v)}]\right) \cdot \left(\prod_{v \in I(\P)} [\C^*]^{|\pi_{\mathrm{st}}^{-1}(v)| - 1}\right) \\&= \left(\prod_{v \in I(\T)} [\calM_{0, \val(v)}]\right) \cdot [\C^*]^{|I(\T)| - |I(\P)|}.
\end{align*}
Note that $|I(\P)|$ is two less than the number of blocks of the ordered partition of $V(\T)$ which is induced by the map $\pi: \hat{\T} \to \P$ as in Lemma \ref{Orderings}. Therefore, summing over all combinatorial types whose stabilized source is equal to $\T$, we get
\[\calM_{\bfx}(\T) = \sum_{\scrP \in \calP_{\bfx}(\T)} \left(\prod_{v \in I(\T)} [\calM_{0, \val(v)}]\right) \cdot [\C^*]^{|I(\T)| - \ell(\scrP) + 2}, \]
as desired.
\end{proof}
Theorem \ref{LocallyConstant} now follows from the framework discussed in this section.
\begin{proof}[Proof of Theorem \ref{LocallyConstant}]
We have
\begin{align*}
    [\Mbar(\bfx)] &= \sum_{\T \in \Gamma_{0,n}}[\calM_{\bfx}(\T)] \\&= \sum_{\T \in \Gamma_{0,n}} \prod_{v \in I(\T)} [\calM_{0, \val(v)}] \sum_{\scrP \in \mathcal{P}_{\bfx}(\T)} [\C^{*}]^{|I(\T)| - \ell(\scrP) + 2}.
\end{align*}
The set $\calP_{\bfx}(\T)$ depends on the partial ordering $\leq_{\bfx}$, and this is determined by the $\bfx$-directing of $\T$. By Lemma \ref{DirectingLocallyConstant}, the $\bfx$-directing of $\T$ is constant as $\bfx$ varies in a resonance chamber, so the proof is complete.
\end{proof}
\section{Sums over trees}\label{TreeSums}
In this section we study the bivariate generating function 
\[\Psi(s, t) = \sum_{k \geq 1} \sum_{n \geq 2} \chi(\Mbar_n(k)) \frac{s^k t^n}{k! n!}  \]
defined in the introduction. Recall that $\Mbar_n = \Mbar(\bfx)$ for \[\bfx = (n, \underbrace{-1, \ldots, -1}_{n \text{ times}}).\]
Thus $\Mbar_n$ parameterizes rubber stable maps \[(X, p_1, \ldots, p_n) \to (Y, 0_{Y}, \infty_Y)\] which are maximally ramified over $0_Y \in Y$ and unramified over $\infty_Y \in Y$. There is a birational morphism $\pi_n: \Mbar_n \to \Mbar_{0, n + 1}$, and $\Mbar_n(k)$ is defined as the preimage under $\pi_n$ of the locus of curves in $\Mbar_{0,n+1}$ which have exactly $k$ irreducible components. We will prove that $\Psi$ satisfies the differential equation given in Theorem \ref{Recursion}:
\begin{equation}\label{DifferentialEquation}
\frac{\partial \Psi}{\partial s} = (1 +t)\left(\log(1 + t) + \exp\left( -\frac{\Psi}{1 + t} \right)\right) + \Psi(\log(1 + t) + 1) - 2t - 1.
\end{equation}
Towards proving (\ref{DifferentialEquation}), it is useful to introduce some auxiliary generating functions. For each $k \geq 1$, define
\[\nu_k(t) \defeq \sum_{n \geq 2} \chi(\Mbar_n(k))\frac{t^n}{n!}. \]
With this notation in place, we have
\[\Psi(s, t) = \sum_{k \geq 1} \nu_k(t) \frac{s^k}{k!}. \]
Equation (\ref{DifferentialEquation}) and hence Theorem \ref{Recursion} will be deduced from Proposition \ref{Bell} below, which makes reference to the \textit{partial exponential Bell polynomials}. To define these, we will use the notation $\lambda \vdash m$ to mean that $\lambda$ is a partition of $m$, and we write $\lambda = (\lambda_1, \ldots, \lambda_m)$ to indicate that $\lambda$ consists of $\lambda_i$ copies of $i$ for $1 \leq i \leq m$, i.e. $\sum i \lambda_i = m$. Moreover, we put $\ell(\lambda) \defeq \sum_{i} \lambda_i$ for the length of $\lambda$. With this notation in hand, the partial exponential Bell polynomial $B_{m, j}$ is a polynomial in the variables $x_1, \ldots, x_{m - j + 1}$, homogeneous of degree $j$, and is defined by the formula
\begin{equation}\label{BellDefn}
B_{m, j}(x_1, \cdots, x_{m - j + 1}) \defeq \sum_{\substack{\lambda \vdash m \\ \ell(\lambda) = j}} \frac{m!}{\prod_{i = 1}^{m} (i!)^{\lambda_i}} \prod_{i = 1}^{m} \frac{x_i^{\lambda_i}}{\lambda_i!}.
\end{equation}
This formula indeed defines a polynomial in variables $x_1, \ldots, x_{m - j + 1}$, because together the conditions that $\lambda \vdash m$ and $\ell(\lambda) = j$ imply that $\lambda_i = 0$ for $i > m - j + 1$. Bell polynomials have a combinatorial interpretation: the coefficient of a monomial $x_{i_1}^{k_{i_1}}\cdots x_{i_n}^{k_{i_n}}$ in $B_{m,j}$ counts the number of ways to partition a set of size $m$ into $j$ blocks, such that there are exactly $k_{i_r}$ blocks of size $i_r$ for each $1 \leq r \leq n$. These polynomials arise in Fa{\`a} di Bruno's formula for the coefficients of the composition of a pair of exponential generating functions ~\cite{Comtet}: if
\[g(x) = \sum_{n \geq 0} \frac{b_n}{n!} x^n \mbox{ and } f(x) = \sum_{n \geq 1} \frac{a_n}{n!} x^n, \]
then
\begin{equation}\label{FaaDiBruno}
    (g \circ f)(x) = b_0 + \sum_{n \geq 1} \frac{\sum_{k = 1}^{n} b_k B_{n,k}(a_1, \ldots, a_{n - k + 1})}{n!} x^n.
\end{equation}
\begin{prop}\label{Bell}
We have
\[\nu_m(t) = \sum_{j = 1}^{m - 1} \nu_1^{(j)}(t) B_{m-1, j}(\nu_1(t), \ldots, \nu_{m - j}(t)) \]
for $m \geq 2$, where $\nu_1^{(j)}(t)$ denotes the $j$th derivative of $\nu_1(t)$.
\end{prop}
Before proving Proposition \ref{Bell}, we pause to show how it implies Theorem \ref{Recursion}.
\begin{proof}[Proof of Theorem \ref{Recursion}]
Observe that
\[\nu_1(t) = \sum_{n \geq 2} \chi(\calM_{0,n+1})\frac{t^n}{n!}. \]
It is well-known that 
\[\nu_1(t) = (1+t)\log(1+t) - t. \]
This can be shown using the fibration $\calM_{0,n+1} \to \calM_{0,n}$ \cite{Manin, McMullen}. Therefore we have $\nu_1'(t) = \log(1 + t)$ and
\[\nu_1^{(j)}(t) = (-1)^{j}\frac{1}{(1 + t)^{j - 1}} \]
for $j \geq 2$. Thus the series
\[G(t,s) \defeq \sum_{j \geq 0} \nu_1^{(j)}(t) \frac{s^j}{j!} \]
satisfies
\begin{align*}
G(t, s) &= (1 + t)\log(1 + t) - t + s\log(1 + t) + \sum_{j \geq 2} (-1)^{j} \frac{1}{(1 + t)^{j-1}} \frac{s^j}{j!} 
\\&= (1 + t + s)\log(1 + t) +(1 + t)\exp\left( -\frac{s}{1 + t} \right) + s - 2t - 1.
\end{align*}
Now we apply Fa{\`a} di Bruno's formula (\ref{FaaDiBruno}) with $a_n = \nu_n(t)$ and $b_n = \nu_1^{(n)}(t)$, together with Proposition \ref{Bell}, to see that
\[G(t, \Psi) = \nu_1(t) + \sum_{n \geq 1} \nu_{n + 1}(t) \frac{s^n}{n!} = \frac{\partial \Psi}{\partial s}. \]
The initial condition on $\Psi$ follows from its definition, so the theorem is proved upon rearranging the formula for $G(t, \Psi)$.
\end{proof}
The remainder of this section is concerned with proving Proposition \ref{Bell}. We first need to establish some preliminary notions. Many of the considerations in the following are inspired by McMullen's expository account ~\cite{McMullen} of prior work of Manin ~\cite{Manin} and Getzler ~\cite{Getzler} on the generating function for Euler characteristics of the moduli spaces $\Mbar_{0, n+1}$. Recall that a \textit{stable tree} is a finite connected tree which has no bivalent vertices. 

\begin{defn}
A \textit{\textbf{rooted tree}} $(T, l)$ is a stable tree $T$ together with a choice of leaf $l \in L(T)$.
\end{defn}

Given an $n$-marked stable tree $\T \in \Gamma_{0,n}$, we can get a rooted tree $(T, l)$ by forgetting all leaf-labellings except for that leaf labelled by $1$. We will also need the notion of a ribbon tree.

\begin{defn}
A \textit{\textbf{ribbon rooted tree}} $(T, l, c)$ is a rooted tree together with a cyclic order $c$ of the half-edges emanating from each vertex. 
\end{defn}

A ribbon structure on a rooted tree $(T, l)$ is equivalent to the data of an embedding of $(T, l)$ in the plane. We will use the notation $\RRT$ for the set of all (isomorphism classes of) ribbon rooted trees, and $\RRT(m) \subset \RRT$ will denote the set of ribbon rooted trees with $m$ internal vertices.

The first step in the proof of Proposition \ref{Bell} is to observe that when $\bfx$ is in the central chamber, the $\bfx$-directing of a tree $\T \in \Gamma_{0,n+1}$ is completely determined by the position of the leaf labelled by $1$. That is, the set $\calP_{\bfx}(\T)$ in Lemma \ref{Orderings} is completely determined by the underlying rooted tree $(T, l)$ obtained from $\T$ by forgetting all but the first marking: see Figure \ref{xdirectingexamplefig}. Given a rooted tree $(T, l)$, we define $o_T$ to be the number of total orderings of $I(T)$ which are compatible with the $\bfx$-directing when $\bfx$ is in the central chamber. We also define $A_1 \defeq 1$, and
\begin{equation}\label{nonfactorialweights}
    A_{d} \defeq \chi(\calM_{0,d}) = (-1)^{d - 1}(d - 3)!
\end{equation}
for $d \geq 3$. For a stable tree $T$, we put
\[a_T \defeq \prod_{v \in V(T)} \frac{A_{\val(v)}}{(\val(v) - 1)!}. \]
We also put
\[N_T \defeq |L(T)| - 1. \]
The following lemma expresses the generating function $\nu_m(t)$ in terms of $\RRT(m)$ and the weights $o_T$, $a_T$ and $N_T$.

\begin{lem}\label{SumOverRRT}
We have
\[\nu_m(t) = \sum_{T \in \RRT(m)} a_T o_T t^{N_T} \]
for $m \geq 1$.
\end{lem}
\begin{proof}
Taking Euler characteristics on both sides of the equation in Proposition \ref{TreeStrataProp}, we see that
\[\chi(\calM_{\bfx}(\T)) = \prod_{v \in I(\T)} \chi(\calM_{0, \val(v)}) \cdot |\{\mathscr{P} \in \calP_{\bfx}(\T) \mid \ell(\mathscr{P}) - 2 = |I(\T)| \}| \]
when $\bfx$ is chosen from the central chamber, and $(T, l)$ is the rooted tree obtained from $\T$ by forgetting all but the first marking. The set
\[\{\mathscr{P} \in \calP_{\bfx}(\T) \mid \ell(\mathscr{P}) - 2 = |I(\T)| \}\]
is in bijection with the set of total orderings of $I(\T)$ which refine $\leq_{\bfx}$: one takes the total ordering induced by the ordering of the non-extremal blocks of $\mathscr{P}$. Since $\ell(\mathscr{P}) - 2 = |I(\T)|$, all such blocks are necessarily singletons. Therefore
\[|\{\mathscr{P} \in \calP_{\bfx}(\T) \mid \ell(\mathscr{P}) - 2 = |I(\T)| \}| = o_T \]
and we can write
\[\nu_m(t) = \sum_{n \geq 2} \sum_{\substack{{\T \in \Gamma_{0, n + 1}}\\ |I(\T)| = m}} o_T \cdot A_{T} \frac{t^{N_T}}{N_T!}, \]
where 
\[A_T \defeq \prod_{v \in V(T)} A_{\val(v)}. \]
Since a rooted tree $(T, l)$ can be made into an element of $\Gamma_{0, n+ 1}$ in $N_T!/|\Aut(T, l)|$ many ways, we can write
\[\nu_m(t) = \sum_{\substack{{(T,l) \text{ rooted tree}}\\ |I(T)| = m}} o_T \cdot A_T \frac{N_T!}{|\Aut(T, l)|} \frac{t^{N_T}}{N_T!}. \]
We now pass to ribbon rooted trees, noting that a rooted tree $T$ has \[\prod_{v \in T}\frac{(\val(v) - 1)!}{|\Aut(T, l)|}\] ribbon structures. Thus
\begin{align*}
    \nu_m(t) &= \sum_{T \in \RRT(m)} o_T \cdot A_T \frac{|\Aut(T, l)|}{\prod_{v \in V(T)}(\val(v) - 1)!} \frac{N_T!}{|\Aut(T,l)|} \frac{t^{N_T}}{N_T!} \\&= \sum_{T \in \RRT(m)} a_T o_T t^{N_T},
\end{align*}
as we wanted to show.
\end{proof}

We will prove Proposition \ref{Bell} by summing over all possible ways of decomposing a ribbon rooted tree into smaller pieces, in a sense we now make precise. Given a rooted tree $(T, l)$, define the \textit{\textbf{base}} of $(T, l)$ to be the unique vertex $b_T \in V(T)$ which is connected by an edge to $l$.
\begin{defn}\label{DecompositionType}
Let $(T, l, c) \in \RRT(m)$. The \textbf{\textit{decomposition type}} of $(T, l, c)$ is the tuple $(k, j, \lambda)$ defined as follows:
\begin{enumerate}[(1)]
    \item $k$ is the number of leaves, besides the root of $T$, which are connected to the base of $T$,
    \item $j = \val(b_T) - k - 1$, and
    \item for each of the $j$ edges $e$ containing $b_T$ which are not ends, one gets a smaller ribbon rooted tree by cutting $e$ in half, and taking the side which does not contain $b_T$. We get a partition $\lambda = (\lambda_1, \ldots, \lambda_{m - 1})$ of $m - 1$ with $\ell(\lambda) = j$ by letting $\lambda_i$ be the number of these smaller trees with exactly $i$ internal vertices.
\end{enumerate}
\end{defn}
See Figure \ref{decompositionexmp} for an example of a decomposition type of a tree in $\RRT(5)$.
\begin{figure}[h]
    \centering
    \includegraphics[scale = 1.25]{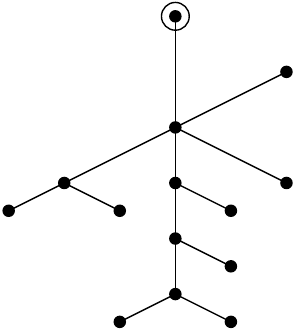}
    \caption{A rooted ribbon tree in $\RRT(5)$ with decomposition type given by $k = j = 2$ and $\lambda = (1, 0, 1, 0)$.}
    \label{decompositionexmp}
\end{figure}
Suppose given $(T, l, c) \in \RRT(m)$ with decomposition type $(k, j, \lambda)$. Then, following part (3) of Definition \ref{DecompositionType}, one gets an ordered tuple $(T_1, \ldots, T_j)$ of rooted ribbon trees such that exactly $\lambda_i$ of them have $i$ internal vertices, where the ordering comes from the ribbon structure at the base vertex $b_T$. We will put $\RRT_\lambda$ for the set of all such ordered tuples. The following lemma expresses $\nu_m(t)$ as a sum over $j, \lambda$ and $\RRT_\lambda$.
\begin{lem}\label{CuttingTheCake}
Suppose $m \geq 2$. Then we have
\[\nu_m(t) = \sum_{j \geq 1}\nu_1^{(j)}(t) \cdot \frac{1}{j!} \sum_{\substack{{\lambda \vdash (m - 1)}\\\ell(\lambda) = j}} \frac{(m-1)!}{\prod_{i = 1}^{m - 1} (i!)^{\lambda_i}} \sum_{(T_1, \ldots, T_j) \in {\RRT_\lambda}} \prod_{i = 1}^{j} a_{T_i} o_{T_i} t^{N_{T_i}}.\]
\end{lem}
\begin{proof}
Suppose we are given $(T, l, c) \in \RRT(m)$ with decomposition type $(k, j, \lambda)$, giving rise to the ordered tuple $(T_1, \ldots, T_j) \in \RRT_\lambda$. Then one has
\[
\sum_{i = 1}^{j} N_{T_i} = N_T - k, \]
while
\[a_T = \prod_{i = 1}^{j} a_{T_i}. \]
 Moreover, we have that 
\[o_T = \frac{(m - 1)!}{\prod_{i = 1}^{m - 1}(i!)^{\lambda_i}} \cdot \prod_{i = 1}^{j} o_{T_i}, \]
since given orderings of $I(T_i)$, we may freely combine them to get an ordering of $I(T)$, and the multinomial coefficient counts the number of ways to do this. Altogether we have
\[a_T o_T t^{N_T} = \frac{(m - 1)!}{\prod_{i = 1}^{m - 1}(i!)^{\lambda_i}} \left(\prod_{i = 1}^{j} a_{T_i} o_{T_i} t^{N_{T_i}}\right) \cdot \frac{\chi(\calM_{0,k + j + 1})}{(k + j)!} t^{k};  \]
the factor of $\chi(\calM_{0,k + j + 1})/{(k + j)!}$ accounts for the base vertex $b_T$.

We now sum over all trees with decomposition type $(k, j, \lambda)$. This amounts to summing over all ways of completing the diagram in Figure \ref{madlib} by making $k$ of the white vertices into leaves, and attaching by their roots an ordered tuple $(T_1, \ldots, T_j) \in \RRT_\lambda$ to the remaining $j$ vertices. The topmost vertex is understood to be the root, and the ribbon structure at the central vertex is determined by the given embedding in the plane. 
\begin{figure}[h]
    \centering
    \includegraphics[scale = 1.25]{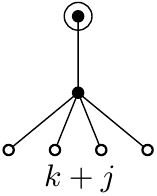}
    \caption{A template for a rooted ribbon tree}
    \label{madlib}
\end{figure}
Proceeding as such, we see that
\begin{align*}
    \nu_m(t) &= \sum_{k, j} \frac{\chi(\calM_{0,k + j + 1})}{(k+j)!} t^{k} \cdot \binom{k + j}{k} \sum_{\substack{{\lambda \vdash (m - 1)}\\\ell(\lambda) = j}} \frac{(m-1)!}{\prod_{i = 1}^{m - 1} (i!)^{\lambda_i}} \sum_{(T_1, \ldots, T_j) \in {\RRT_\lambda}} \prod_{i = 1}^{j} a_{T_i} o_{T_i} t^{N_{T_i}} \\&= \sum_{k, j} \frac{\chi(\calM_{0,k + j + 1})}{k!} t^{k} \cdot \frac{1}{j!} \sum_{\substack{{\lambda \vdash (m - 1)}\\\ell(\lambda) = j}} \frac{(m-1)!}{\prod_{i = 1}^{m - 1} (i!)^{\lambda_i}} \sum_{(T_1, \ldots, T_j) \in {\RRT_\lambda}} \prod_{i = 1}^{j} a_{T_i} o_{T_i} t^{N_{T_i}},
\end{align*}
where the binomial coefficient $\binom{k+j}{k}$ accounts for the choice of $k$ white vertices to turn into leaves. Due to the stability condition and the fact that $m \geq 2$, we only have ribbon trees of decomposition type $(k, j, \lambda)$ when $k \geq 0$, $j \geq 1$, and $k + j \geq 2$. When $j = 1$, we have
\[ \sum_{k \geq 1} \frac{\chi(\calM_{0,k + j + 1})}{k!} t^{k} =  \sum_{k \geq 1} \frac{\chi(\calM_{0,k + 2})}{k!} t^{k} = \nu_1'(t), \]
and for $j \geq 2$, we have 
\[ \sum_{k \geq 0} \frac{\chi(\calM_{0,k + j + 1})}{k!} t^{k} = \nu_1^{(j)}(t), \]
which completes the proof.
\end{proof}

We are now ready to prove Proposition \ref{Bell}.
\begin{proof}[Proof of Proposition \ref{Bell}]
We claim that for any partition $\lambda \vdash (m - 1)$ with $\ell(\lambda) = j$, we have
\begin{equation}\label{CompareCoeffs}
\frac{1}{j!} \sum_{(T_1, \ldots, T_j) \in {\RRT_\lambda}} \prod_{i = 1}^{j} a_{T_i} o_{T_i} t^{N_{T_i}} = \prod_{i = 1}^{m - 1} \frac{\nu_i(t)^{\lambda_i}}{\lambda_i!}.
\end{equation}
Indeed, fix a multiset of stable ribbon rooted trees $S = \{T_1, \ldots, T_j\}$ such that $S$ contains $\lambda_i$ trees with $i$ internal vertices for $1 \leq i \leq m - 1$; we will calculate the coefficient of $\prod_{i = 1}^{j} a_{T_i} o_{T_i} t^{N_{T_i}}$ on both sides. Suppose that, removing duplicates from $S$, we end up with the set $\{J_1, \ldots, J_r\}$, and the tree $J_i$ occurs $k_i$ times in $S$, so $\sum_{i = 1}^{r} k_i = j$. Then there are $ j! / (k_1!\cdots k_r!)$ ways to order $S$, so the term $\prod_{i = 1}^{j} a_{T_i} o_{T_i} t^{N_{T_i}}$ occurs with coefficient $1/{(k_1!\cdots k_r!)}$ on the left-hand side. Towards calculating on the right-hand side, suppose that the $\lambda_i$ elements of $S$ with $i$ internal vertices are chosen from a set $L_i = \{J_{i1}, \ldots, J_{ir_{\lambda_i}} \} \subset \RRT(i)$ of trees such that $J_{is}$ appears $d_{is}$ times. Then the term $\prod_{i = 1}^{j} a_{T_i} o_{T_i} t^{N_{T_i}}$ occurs on the right-hand side with coefficient
\[\frac{1}{\prod_{i = 1}^{m-1} \lambda_i!} \prod_{i = 1}^{m - 1} \frac{\lambda_i!}{d_{i1}!\cdots d_{ir_{\lambda_i}}!} = \frac{1}{\prod_{i = 1}^{m - 1} \prod_{c = 1}^{r_{\lambda_i}} d_{ic}!}.  \]
By the way we defined the integers $k_i$ and $d_{ic}$, we have an equality of multisets
\[\{k_i \mid 1 \leq i \leq r\} = \{d_{ic} \mid 1 \leq i \leq m-1, 1 \leq c \leq r_{\lambda_i} \},  \]
so (\ref{CompareCoeffs}) indeed holds. Proposition \ref{Bell} follows upon combining this with Lemma \ref{CuttingTheCake} and recalling the formula (\ref{BellDefn}) for the Bell polynomials.
\end{proof}

\section{Further directions}\label{NewHorizons}
The preceding results suggest many further questions about the topology of $\Mbar(\bfx)$, and we record some of these here.
\subsection{Wall-crossing} Theorems \ref{Recursion} and \ref{LocallyConstant} give two out of three steps in a potential program to determine the Euler characteristic of $\Mbar(\bfx)$ for arbitrary choice of $\bfx$. The third step in this program would be a \textit{wall-crossing formula}:  a tractable formula for the difference
\[ [\Mbar(\bfx)] - [\Mbar(\bfy)] \in K_0(\mathsf{Var}/\C) \]
when $\bfx$ and $\bfy$ are chosen from adjacent resonance chambers, say separated by the wall $W_{S}$ for some $S \subset [n]$. Given $\T \in \Gamma_{0,n}$ and $e \in E(\T)$, say that $e$ is an \textbf{\textit{$(S, S^c)$-split}} if upon deleting $e$ from $T$, one connected component supports the leaves indexed by $S$, while the other supports the leaves indexed by $S^c$. From Proposition \ref{TreeStrataProp}, it follows that
\begin{align*}
&[\Mbar(\bfx)] - [\Mbar(\bfy)] = \sum_{\substack{{\T \in \Gamma_{0,n}} \\{\T \text{ has an }(S, S^c)-\text{split}}}} [\calM_{\bfx}(\T)] - [\calM_{\bfy}(\T)] \\&= \sum_{\substack{{\T \in \Gamma_{0,n}} \\{\T \text{ has an }(S, S^c)-\text{split}}}} \prod_{v \in I(\T)} [\calM_{0, \val(v)}] \left(\sum_{\mathscr{P} \in \calP_{\bfx}(\T)}[\C^*]^{I(\T) - \ell(\mathscr{P}) + 2} - \sum_{\mathscr{P}' \in \calP_{\bfy}(\T)}[\C^*]^{I(\T) - \ell(\mathscr{P}') + 2} \right),
\end{align*}
but we do not see how to use this description to derive a useful formula. One may hope that the difference $[\Mbar(\bfx)] - [\Mbar(\bfy)]$ is expressed in terms of classes $[\Mbar(\bfx')]$ as $\bfx'$ varies over a set of smaller ramification data, as is the case for Hurwitz numbers ~\cite{SSV, ChamberStructure}. Indeed one has
\[\sum_{\substack{{\T \in \Gamma_{0,n}} \\{\T \text{ has an }(S, S^c)-\text{split}}}} \prod_{v \in I(\T)} [\calM_{0, \val(v)}] = [\Mbar_{0, |S| + 1}] [\Mbar_{0, |S^c| + 1}], \]
but it is unclear how to control the terms indexed by $\calP_{\bfx}(\T)$ and $\calP_{\bfy}(\T)$.

\subsection{Equivariant Euler characteristics}
The space $\Mbar_n$ admits a natural action of $S_n$ by permuting the last $n$ marked points on the source curve, and it is natural to ask for the \textit{equivariant} Euler characteristic
\[\chi^{S_n}(\Mbar_n) \defeq \sum_{i} (-1)^{i} \ch_n H^i(\Mbar_n; \Q) \in \Lambda. \]
Here
\[\Lambda \defeq \lim_{\longleftarrow} \Q[x_1, \ldots, x_n]^{S_n} \]
is the ring of symmetric functions, and for an $S_n$-representation $W$ with a decomposition
\[W = \bigoplus_{\lambda \vdash n} V_\lambda^{\oplus a_\lambda} \]
into irreducibles, we define
\[\ch_n W \defeq \sum_{\lambda \vdash n} a_{\lambda} s_{\lambda}, \]
where $s_{\lambda} \in \Lambda$ is the Schur function corresponding to the partition $\lambda$. A recursive algorithm for the calculation of $\chi^{S_n}(\Mbar_{0,n+1})$ is due to Getzler ~\cite{Getzler}, and uses the theory of operads (indeed, his techniques also give the $S_n$-equivariant Po{\'i}ncare polynomials). Our formula for the numerical Euler characteristic $\chi(\Mbar_n)$ relies on Proposition \ref{TreeStrataProp}, which is manifestly non-equivariant. It is intriguing to ask whether there is a natural operadic framework in which to consider the spaces $\Mbar_{n}$, which would allow for their $S_n$-equivariant study. 

\subsection{Asymptotic behavior of \texorpdfstring{$\chi(\Mbar_{n})$}{chi}} Manin was the first to study the generating function for the numbers $\chi(\Mbar_{0,n+1})$ as a sum over all stable trees in ~\cite{Manin}. He states the asymptotic formula
\[\chi(\Mbar_{0,n+1}) \sim \frac{1}{\sqrt{n}}\left(\frac{n}{e^2 - 2e} \right)^{n - \frac{1}{2}}. \]
A rigorous proof of this formula has been given by Readdy ~\cite{Readdy}. It would be interesting to see whether her techniques may be combined with Theorem \ref{Recursion} to derive an asymptotic for $\chi(\Mbar_{n})$. As Table \ref{EulerChars} indicates, it seems that $\chi(\Mbar_{n})$ grows much faster than $\chi(\Mbar_{0, n + 1})$. We thus make the following conjecture.
\begin{conj}\label{complexity}
We have 
\[\lim_{n \to \infty} \frac{\chi(\Mbar_{0,n+1})}{\chi(\Mbar_n)} = 0. \]
\end{conj}
It is discussed in \cite{CMR} that $\Mbar(\bfx)$ may be obtained as an iterated blow-up of $\Mbar_{0,n}$. The results in ~\cite{MaximalContacts} suggest that in the case $\Mbar(\bfx) = \Mbar_n$, this sequence of blow-ups can be given an explicit algorithmic description. Conjecture \ref{complexity} suggests that the complexity of this iterated blow-up increases dramatically with $n$.

\subsection{Chern numbers}
The Euler characteristic of $\Mbar(\bfx)$ is an example of a Chern number, as it is equal to the degree of the zero-dimensional part of the Chern-Schwartz-MacPherson class:
\[\chi(\Mbar(\bfx)) = \mathrm{deg}\,c_0^{\mathrm{SM}}(\Mbar(\bfx)). \]
Theorem \ref{LocallyConstant} implies that this number is constant on the resonance chambers. It is interesting to ask whether the same is true for general Chern numbers, i.e. the numbers
\[ \mathrm{deg}\,h(c_0^{\mathrm{SM}}(\Mbar({\bfx})), \ldots, c_{n - 3}^{\mathrm{SM}}(\Mbar({\bfx}))) \in \Q \]
where $h(t_0, \ldots, t_{n - 3})$ is a homogeneous polynomial of degree $n - 3$, when $t_i$ is understood to have degree $ n - 3 - i$. For example, when $h = t_1t_{n-4}$, the corresponding Chern numbers can be expressed in terms of the Hodge numbers ~\cite{LibgoberWood}, so these numbers are also constant in the resonance chambers by Theorem \ref{LocallyConstant}. Motivated by these cases and the piecewise polynomiality of Hurwitz numbers ~\cite{GJV, SSV, ChamberStructure}, we conclude with the following conjecture.
\begin{conj}\label{ChernConjecture}
Fix $h(t_0, \ldots, t_{n-3})$ a homogeneous polynomial of degree $n-3$, where $t_i$ has degree $n - 3 - i$. Then, the function $p : A_n \cap \Z^n \to \Q$ defined by
\[p(\bfx) = \mathrm{deg}\,h(c_0^{\mathrm{SM}}(\Mbar({\bfx})), \ldots, c_{n - 3}^{\mathrm{SM}}(\Mbar({\bfx}))) \]
is piecewise polynomial. Moreover, the chambers of polynomiality are given by the resonance chambers.
\end{conj}

\bibliographystyle{alpha}
\bibliography{biblio}
\end{document}